\documentclass[9pt,reqno]{amsart}

 \usepackage{mathrsfs}

\usepackage{color}
\usepackage{amsmath, amsthm, amssymb}
\usepackage{amsfonts}
\usepackage[dvips]{epsfig}
\usepackage{graphicx}
\usepackage{caption}
\usepackage{subcaption}
\usepackage[english]{babel}
\usepackage[left=4.5cm,right=4.5cm,top=3cm,bottom=3cm]{geometry}
\usepackage{hyperref}

\usepackage{tikz}
\usepackage{rotating}
\usepackage[utf8]{inputenc}
\usepackage{cite}
\usepackage{amscd}
\usepackage{color}
\usepackage{bm}
\usepackage{enumerate}
\usepackage{verbatim}
\usepackage{hyperref}
\usepackage{amstext}
\usepackage{latexsym}
\theoremstyle{plain}
\newtheorem{theorem}{Theorem}[section]

\newtheorem{proposition}[theorem]{Proposition}
\newtheorem{lemma}[theorem]{Lemma}

\numberwithin{theorem}{section}
\numberwithin{equation}{section}

\newcommand{\average}{{\mathchoice {\kern1ex\vcenter{\hrule height.4pt
width 6pt depth0pt} \kern-9.7pt} {\kern1ex\vcenter{\hrule
height.4pt width 4.3pt depth0pt} \kern-7pt} {} {} }}

\def\R{\mathbb{R}}

\renewcommand{\b }{\beta }
\renewcommand{\d}{\delta }

\newcommand{\D }{\Delta }

\newcommand{\e }{\varepsilon }

\newcommand{\G }{\Gamma}
\renewcommand{\l }{\lambda }

\newcommand{\n }{\nabla }

\renewcommand{\phi}{\varphi}

\newcommand{\s }{\sigma }

\renewcommand{\O }{\Omega }

\newcommand{\ov}{\overline}

\newcommand{\be}{\begin{equation}}
\newcommand{\ee}{\end{equation}}

\newcommand{\de}{\partial}

\newcommand{\ti}{\widetilde}

\newcommand{\calO }{\mathcal{O}}

\newcommand{\calD }{\mathcal{D}}

\newcommand{\N}{\mathbb{N}}

\newcommand{\cC}{{\mathcal C}}

\newcommand{\cR}{{\mathcal R}}

\newcommand{\B}{{Q}}

\renewcommand{\epsilon}{\varepsilon}

\begin{document}
 
\date{\today}
\title[Mass effect on an elliptic PDE involving two Hardy-Sobolev critical exponents]{Mass effect on an elliptic PDE involving two Hardy-Sobolev critical exponents}
\author{El Hadji Abdoulaye THIAM}
\address{H. E. A. T. : Université Iba Der Thiam de Thies, UFR des Sciences et Techniques, département de mathématiques, Thies.}
\email{elhadjiabdoulaye.thiam@univ-thies.sn}
\begin{abstract}
We let $\Omega$ be a bounded domain  of $\mathbb{R}^3$ and $\Gamma$ be  a closed curve contained in $\Omega$.  We study existence of positive solutions $u \in H^1_0\left(\Omega\right)$ to the  equation
$$
-\Delta u+hu=\l\rho^{-s_1}_\Gamma u^{5-2s_1}+\rho^{-s_2}_\Gamma u^{5-2s_2} \qquad \textrm{ in } \Omega
$$
where $h$ is a continuous function and $\rho_\Gamma$ is the distance function to $\Gamma$. We prove existence of solutions depending on the regular part of the Green function of linear operator. We prove the existence of positive mountain pass solutions for this Euler-Lagrange equation depending on the mass which is the regular part of the Green function of the linear operator $-\D+h$.
\end{abstract}
\address{H. E. A. T. : Université Iba Der Thiam de Thies, UFR des Sciences et Techniques, département de mathématiques, Thies.}
\email{elhadjiabdoulaye.thiam@univ-thies.sn}
\maketitle
\section{Introduction}
In this paper, we are concerned with the mass effect on the existence of mountain pass solutions of the following nonlinear partial differential equation involving two Hardy-Sobolev critical exponents in $\R^3$. More precisely, letting $h$ be a continuous function and $\l$ be a real parameter, we consider
\begin{equation}\label{Euler-Lagrange11}
\begin{cases}
\displaystyle-\Delta u(x)+ h u(x)=\l \frac{u^{5-2s_1}(x)}{\rho_\G^{s_1}(x)}+\frac{u^{5-2s_2}(x)}{\rho_\G^{s_2}(x)} \qquad & \textrm{ in $\O$}\\\\
u(x)>0 \qquad \textrm{ and } \qquad u(x)=0 &\textrm{ on $\de \Omega$},
\end{cases}
\end{equation}
where $\rho_\G(x):=\inf_{y \in \Gamma}|y-x|$ is the distance function to the curve $\Gamma$ and for  $0< s_2 <s_1<2$, $2^*_{s_1}:=6-2s_1$ and $2^*_{s_2}:=6-2s_2$ are two critical Hardy-Sobolev exponents.

To study the equation  \eqref{Euler-Lagrange11}, we consider the following non-linear functional $\Psi: H^1_0(\O) \to \R$ defined by:
\begin{equation}\label{Functional}
\Psi(u):=\frac{1}{2} \int_\O |\n u|^2 dx+\frac{1}{2} \int_\O h(x) u^2 dx- \frac{\l}{2^*_{s_1}} \int_\O \rho_\G^{-s_1}(x) |u|^{2^*_{s_1}} dx-\frac{1}{2^*_{s_2}} \int_\O \rho_\G^{-s_2}(x) |u|^{2^*_{s_2}} dx.
\end{equation}
Then there exists a positive constant $r>0$ and $u_0 \in H^1_0(\Omega)$ such that $\|u_0\|_{H^1_0(\O)}>r$ and
$$
\inf_{\|u\|_{H^1_0(\O)}=r} \Psi(u) >\Psi(0) \geq \Phi(u_0), 
$$
see for instance the paper of the author [\cite{Thiam1}, Lemma 4.5]. 
Then the point $(0, \Psi(0))$ is separated from the point $(u_0, \Psi(u_0))$ by a ring of mountains. Set
\begin{equation}\label{Heat}
c^*:=\inf_{P \in \mathcal{P}} \max_{v \in P} \Psi(v),
\end{equation}
where $\mathcal{P}$ is the class of continuous paths in $H^1_0(\O)$ connecting $0$ to $u_0$. Since $2^*_{s_2}>2^*_{s_1}$, the function $t \longmapsto \Psi(tv)$ has the unique maximum for $t \geq 0$. Furthermore, we have
$$
c^*:=\inf_{u \in H^1_0(\O), u \geq 0, u \neq 0} \max_{t \geq 0} \Psi(tu).
$$
Due to the fact that the embedding of $H^1_0(\O)$ into the weighted Lebesgue spaces $L^{2^*_{si}}(\rho_\Gamma^{-si} dx)$ is not compact, the functional $\Psi$  does not satisfy the Palais-Smale condition. Therefore, in general $c^*$ might not be a critical value for $\Psi$.

To recover compactness, we study the following non-linear problem: let $x=(y, z) \in \R\times \R^{2}$ and consider 
\begin{equation}\label{A1A1}
\begin{cases}
\displaystyle-\Delta u=\l\frac{u^{2^*_{s_1}-1}(x)}{|z|^{s_1}}+\frac{u^{2^*_{s_2}-1}}{|z|^{s_2}} \qquad & \textrm{ in $\R^3$}\\\
u(x)>0 & \textrm{ in $\R^3$}.
\end{cases}
\end{equation}
To obtain solutions of \eqref{A1A1}, we consider the functional $\Phi: \calD^{1,2}(\R^N)$ defined by
\begin{equation}\label{Functional1}
\Phi(u):=\frac{1}{2} \int_{\R^3} |\n u|^2 dx- \frac{\l}{2^*_{s_1}} \int_{\R^3} |z|^{-s_1}|u|^{2^*_{s_1}} dx-\frac{1}{2^*_{s_2}} \int_{\R^3} |z|^{-s_2}|u|^{2^*_{s_2}} dx.
\end{equation}
Next, we define
$$
\beta^*:=\inf_{u \in D^{1, 2}(\R^3), u \geq 0, u \neq 0} \max_{t \geq 0} \Phi(tu).
$$
Then we get compactness provided
$$
c^*<\beta^*,
$$
see Proposition 4.3 in \cite{Thiam1}. Therefore the existence, symmetry and decay estimates of non-trivial solution $w\in \calD^{1,2}(\R^3)$ of \eqref{A1A1} play an important role in problem \eqref{Euler-Lagrange11}. Then we have the following results.
\begin{proposition}\label{TheoremA}
Let $0 \leq s_2<s_1<2$, $\l \in \R$. Then equation
\begin{equation}\label{Euler-Lagrange1}
\begin{cases}
\displaystyle-\Delta u=\l\frac{u^{2^*_{s_1}-1}(x)}{|z|^{s_1}}+\frac{u^{2^*_{s_2}-1}}{|z|^{s_2}} \qquad & \textrm{ in $\R^3$}\\\
u(x)>0 & \textrm{ in $\R^3$}
\end{cases}
\end{equation}
has a positive ground state solution $w \in \calD^{1,2}(\R^3)$ depending only on $|y|$ and $|z|$. 
Moreover
\begin{equation}\label{eq:up-low-bound-w-3D}
\frac{C_1}{1+|x|} \leq w(x) \leq \frac{C_2}{1+|x|} \qquad \textrm{ in   $\R^3$}.
\end{equation}
Moreover,  for   $|x|= |(t,z)|\leq 1$, we have 
\begin{equation}\label{eq:up-low-bound-w-3D1}
|\n w (x)|+ |x| |D^2 w (x)|\leq C_2 |z|^{1-s_1}
\end{equation}
and if   $|x|= |(t,z)|\geq 1$, we have 
\begin{equation}\label{eq:up-low-bound-w-3D2}
|\n w(x)|+ |x| |D^2 w(x)|\leq C_2 \max(1, |z|^{-s_1})|x|^{1 -N}.
\end{equation}
\end{proposition}

Next, we let   $G(x,y)$ be  the Dirichlet Green function of the operator $-\D +h$, with zero Dirichlet data. It satisfies 
\be\label{eq:Green-expan-introduction}
\begin{cases}
-\D_x G(x,y)+h(x) G(x,y)=0&  \qquad\textrm{  for every $x\in \O\setminus\{y\}$}\\
G(x,y)=0 &  \qquad\textrm{  for every $x\in\de  \O$.}
\end{cases}
\ee
In addition there exists a continuous function   $\textbf{m}:\O\to \R$ and a positive constant $c>0$ such that  
\be \label{eq:expans-Green}
 G (x,y)=\frac{c}{  |x- y|}+ c\, \textbf{m}(y)+o(1)  \qquad \textrm{ as $x \to y.$} 
\ee
We call the   function   $\textbf{m}:\O\to \R$  the    \textit{mass} of $-\D+h$ in $\O$.   We note that $-\textbf{m}$ is occasionally called the \textit{Robin function} of $-\D+h$ in the literature. Then our main result is the following.
Then we have
\begin{theorem}\label{th:main1}
Let $0\leq s_2 <s_1<2$ and   $\O$  be a   bounded domain of $\R^3$. Consider  $\G$ a smooth closed curve contained in $\O$.
Let  $h$ be  a continuous function such that the linear operator $-\D+h$ is coercive. We assume that there exists $y_0 \in \Gamma$ such that 
\begin{equation}\label{ExistenceAssumption}
m(y_0)>0.
\end{equation}
Moreover there exists $u \in H^1_0(\Omega)\setminus \lbrace 0\rbrace$ non-negative solution of
$$
-\Delta u(x)+ h u(x)=\l \frac{u^{5-2s_1}(x)}{\rho_\G^{s_1}(x)}+\frac{u^{5-2s_2}(x)}{\rho_\G^{s_2}(x)} \qquad \textrm{ in $\O$}.
$$
\end{theorem}
In contrast to the case $N \geq 4$ (see \cite{Thiam1} for more details), the existence of solution does not depend on the local geometry of the singularity but on the location of the curve $\Gamma$.  Besides in the study of Hardy-Sobolev equations in domains with interior singularity for the Three dimensional case, the effect of the mass plays an important role in the existence of positive solutions. For Hardy-Sobolev inequality on Riemannian manifolds with singularity a point, Jaber \cite{Jaber1} proved the existence of positive solutions when the mass is positive. We refer also to \cite{Jaber2} for existence of mountain pass solution to a Hardy-Sobolev equation with an additional perturbation term. For the Hardy-Sobolev equations on domains with singularity a curve, we refer to the papers of the author and Fall \cite{Fall-Thiam} and the author and  Ijaodoro \cite{Esther}. We also suggest to the interested readers the nice work of Schoen-Yau \cite{Schoen-Yau} and \cite{Schoen-Yau1} for more details related to the positive mass theorem. We also mention that this paper is the 3-dimensional version of the work of thye author \cite{Thiam1}.\\

The proof of Theorem \ref{th:main1} relies on test function methods. Namely we build appropriate test functions allowing to compare $c^*$ and $\b^*$. Near the concentration point $y_0 \in \Gamma$, the test function is similar to the test function in the case $N \geq 4$ but away from it is replaced with the regular part of the Green function which makes apear the mass, see  Section \ref{ss:proof-of-th-3D}.
%
%
%
%
%
%
%
%
%
%
%
%
%
%
%
%
%
%
%
%
%
%
%
%
%
%
%
%
%
%
%
%
\section{Tool Box}\label{s:3D-case}
We consider the function
$$
\cR:\R^3\setminus\{0\}\to \R,    \qquad x\mapsto  \cR(x)=\frac{1}{|x|}
$$
which  satisfies
\begin{equation}\label{eq:Green-R3-3D}
-\D \cR=0 \qquad \textrm{ in $\R^3\setminus\{0\}$. }  
\end{equation}
We denote by $G$ the solution to the equation 
\begin{equation}\label{eq:Green-3D}
\begin{cases}
-\D_x G(y, \cdot)+h G(y,\cdot)=0& \qquad \textrm{  in $\O\setminus \{y\}$. }  \\
G(y,\cdot )=0&   \qquad \textrm{  on $\de \O $, }
\end{cases}
\end{equation}
and satisfying
\be\label{eq:expand-Green-trace}
G(x,y)=  \cR(x-y)+O(1)\qquad\textrm{ for $x, y\in \O$ and $x\not= y$.}
\ee
We note that $G$ is proportional to the Green function of $-\D+h$ with zero Dirichlet data.\\
We let $\chi\in C^\infty_c(-2,2)$ with $\chi \equiv 1$ on $(-1,1)$ and $0\leq \chi<1$. For $r>0$,   we  consider the cylindrical symmetric cut-off function
\be\label{eq:def-cut-off-cylind} 
\eta_r(t,z)=\chi\left(\frac{|t|+|z|}{r} \right) \qquad \qquad\textrm{ for every  $(t,z)\in \R\times \R^2$}.
\ee
It is clear that 
$$
\eta_r\equiv 1\quad \textrm{ in $\B_r$},\qquad \eta_r\in H^1_0({Q}_{2r}),\qquad |\n \eta_r|\leq  \frac{C}{r} \quad\textrm{ in $\R^3$}.
$$
For $y_0\in \O$, we let  $r_0\in (0,1)$  such that   
\be\label{eq:def-r0} 
y_0+ Q_{2r_0}\subset\O. 
 \ee  
We define the function $M_{y_0}: Q_{2r_0}\to \R$ given by
\begin{equation}\label{C9}
M_{y_0}(x):=  G (y_0,x+y_0)-{\eta_r}(x)\frac{1}{|x|}     \qquad \textrm{ for every $x\in Q_{2r_0}$}.
\end{equation}
 It follows from  \eqref{eq:expand-Green-trace} that  $M_{y_0}\in  L^\infty(Q_{r_0})$. By \eqref{eq:Green-3D} and \eqref{eq:Green-R3-3D}, 
 $$
|-\D {M}_{y_0}(x)+h(x) {M}_{y_0}(x)|\leq \frac{C}{|x|}= C \cR(x) \qquad \textrm{ for every  $x\in Q_{r_0}$},
 $$
 whereas $\cR\in L^p( Q_{r_0})$ for every $p\in (1,3)$. Hence by
elliptic regularity theory, $M_{y_0}\in W^{2,p}(Q_{r_0/2})$ for every $p\in (1,3)$. Therefore by Morrey's embdding theorem, we deduce that 
\be \label{eq:regul-beta}
\|M_{y_0}\|_ {C^{1,\varrho}(Q_{r_0/2})}\leq C \qquad \textrm{ for every $\varrho\in (0,1)$.}
\ee
In view of \eqref{eq:expans-Green}, the mass of the operator $-\D+h$ in $\O$ at the point $y_0\in   \O$ is given by  
\be \label{eq:def-mass}
  \textbf{m}(y_0)={M}_{y_0}(0).
\ee
Next, we have the following result which will be important in the sequel.
\begin{lemma}\label{lem:v-to-cR}
Consider the function $v_\e:  \R^3\setminus\{0\}\to \R$ given by 
$$
v_\e(x)= \e^{-1} w\left(\frac{x}{\e}\right).
$$
Then there exists a   constant $\textbf{c}>0$ and a sequence $(\e_n)_{n\in \N}$    (still denoted by $\e$)  such that 
$$
v_\e (x) \to  \frac{\textbf{c}}{|x|}  \qquad \textrm{ and } \qquad \n v_\e (x) \to -\textbf{c}  \frac{x}{|x|^3} \qquad \textrm{ for  all most every  $x\in \R^3 $ } 
$$
and 
\be\label{eq:nv-eps-to-nv-C1}
v_\e (x) \to  \frac{\textbf{c}}{|x|}  \qquad \textrm{ and } \qquad  \n v_\e (x) \to    -\textbf{c}  \frac{x}{|x|^3} \qquad \textrm{ for  every   $x\in \R^3\setminus\{z=0\}$.  } 
\ee
\end{lemma}
\begin{proof}
By Proposition \ref{TheoremA}, we have that  $(v_\e)$ is bounded in $C^2_{loc}(\R^3\setminus\{z=0\})$. Therefore by Arzel\'a-Ascolli's theorem $v_\e$ converges to $v$ in $C^1_{loc}(\R^3\setminus\{z=0\})$. In particular,
$$
v_\e \to v  \qquad \textrm{ and } \qquad \n v_\e \to \n v \qquad\textrm{ almost every where on $\R^3$.}
$$  
It is plain, from \eqref{eq:up-low-bound-w-3D}, that 
\be\label{eq:up-low-bound-v-eps-3D}
 0<\frac{C_1}{\e+|x|} \leq v_\e(x) \leq \frac{C_2}{\e+ |x|} \qquad \textrm{ for almost every $x\in   \R^3 $.}
\ee
By \eqref{A1A1}, we have
\begin{equation}\label{eq:eq-for-v-eps}
-\D v_{\e}(x)=\l {\e}^{2-s_1} \frac{v_\e^{5-2s_1}(x)}{|z|^{s_1}}+{\e}^{2-s_2} \frac{v_\e^{5-2s_2}(x)}{|z|^{s_2}} \qquad \textrm{ in } \R^3.
\end{equation}
Newt, we let $\varphi \in C^\infty_c\left(\R^3\setminus \lbrace 0\rbrace\right)$. We multiply \eqref{eq:eq-for-v-eps} by $\varphi$ and integrate by parts to get
$$
-\int_{\R^3} v_{\e} \D \varphi dx= \l {\e}^{2-s_1} \int_{\R^3} \frac{v_\e^{5-2s_1}(x)}{|z|^{s_1}} \varphi(x) dx+{\e}^{2-s_2} \int_{\R^3} \frac{v_\e^{5-2s_2}(x)}{|z|^{s_2}} \varphi(x) dx.
$$
By \eqref{eq:up-low-bound-v-eps-3D} and the dominated convergence theorem, we can pass to the limit in the above identity and deduce that 
$$
\Delta v=0 \qquad \quad\textrm{ in } \calD^\prime\left(\R^3\setminus \lbrace 0\rbrace\right).
$$
In particular $v$ is equivalent to a function of class $C^\infty\left(\R^3 \setminus \lbrace0\rbrace\right)$ which is still denoted by $v$. Thanks to \eqref{eq:up-low-bound-v-eps-3D}, by B\^{o}cher's theorem, there exists a constant $\textbf{c}>0$  such that
$
v(x)=\frac{\textbf{c}}{|x|}.
$
The proof of the lemma is thus finished.
\end{proof}
We finish this section by the following estimates. Thanks to the decay estimates in Proposition \ref{TheoremA}, we have
\begin{lemma}\label{lem:estimates-for-3D}
There exists a constant $C>0$ such that for every  $\e,r\in (0,r_0/2)$ and for $s \in (0,2)$, we have 
\be\label{eq:est-nw-sq-3D}
  \int_{  \B_{r/\e}}   |\n w |^2 dx\leq C\max\left(1, \frac{\e}{r}\right),\qquad    \int_{  \B_{r/\e}}   | w |^2 dx\leq C \max\left(1 , \frac{r}{\e} \right), 
\ee
\be \label{eq:est-wnw-sq-3D} 
   \int_{  \B_{r/\e}}  w  |\n w|   dx\leq C \max\left(  1, \log\frac{r}{\e}\right),
\ee
\be \label{eq:est-nw-3D}
  \int_{ \B_{r/\e}}   |\n  w| dx \leq C \max\left(1 , \frac{r}{\e} \right), \qquad   \int_{ \B_{r/\e}}   |  w| dx \leq  C \max\left(1 , \frac{r^2}{\e^2} \right)   
\ee
%
and 
\be\label{eq:est-L2-star-3D}
  \e^2 \int_{\B_{r/\e}} |z|^{-s}|x|^2  w ^{2^*_s}  dx+\e   \int_{\B_{4r/\e}\setminus \B_{r/\e}} |z|^{-s}  w ^{2^*_s-1}  dx+  \int_{\R^3\setminus \B_{r/\e}} |z|^{-s}  w ^{2^*_s}  dx=o(\e).
\ee
\end{lemma}
\section{Proof of the main result}\label{ss:proof-of-th-3D}
Given $y_0\in \G\subset\O\subset \R^3$, we let $r_0$ as defined in \eqref{eq:def-r0}. For    $r\in (0, r_0/2)$,  we consider $F_{y_0}: Q_r\to \O$ parameterizing a neighborhood of $y_0$ in $\O$, with the property that $F_{y_0}(0)=y_0$,
\begin{equation}\label{Sym-dist}
\rho_\Gamma(F_{y_0}(x))= |z|, \qquad \textrm{ for all $x=(y, z)\in Q_r$}.
\end{equation}
Moreover in these local coordinates, we have
\begin{equation}\label{Metric}
g_{ij}(x)= \d_{ij}+O(|x|)
\end{equation}
and
\begin{equation}\label{DeterminantMetric}
\sqrt{|g|}(x)=1+\langle A, z\rangle+O\left(|x|^2\right),
\end{equation}
where $A\in \R^2$ is the vector curvature of $\Gamma$ and  $|g|$ stands for the determinant of $g$, see \cite{Fall-Thiam} for more details related to this parametrization.\\

Next, for   $\e>0$,  we consider  $u_\e: \O \to  \R$ given  by
$$
u_\e(y):=\e^{-1/2} \eta_r(F^{-1}_{y_0}(y)) w \left(\frac{F^{-1}_{y_0}(y)}{\e} \right).
$$
We can now define the test function $\Psi_\e:\O\to \R$ by
\be 
\label{eq:TestFunction-Om-3D}
\Psi_\e\left(y\right)=u_\e(y)+\e^{1/2}  \textbf{c}\,  \eta_{2r}(F^{-1}_{y_0}(y) ){M}_{y_0}(F^{-1}_{y_0}(y) ).
\ee
It is plain that $\Psi_\e\in H^1_0(\O)$ and 
$$
\Psi_\e\left(F_{y_0}(x)\right)=\e^{-1/2} \eta_r(x) w \left(\frac{x}{\e} \right)+\e^{1/2}  \textbf{c} \,  \eta_{2r}(x) {M}_{y_0}(x) \qquad\textrm{ for every $x\in \R^N$.}
$$
To alleviate the notations, we will write $\e$ instead of $\e_n$ and  we will remove the subscript  $y_0$, by writing $M$ and $F$ in the place of ${M}_{y_0}$ and $F_{y_0}$ respectively.
We define 
$$
\ti \eta_r(y):=\eta_r(F^{-1}(y)),\qquad V_\e(y):=v_\e(F^{-1}(y)) \qquad\textrm{ and } \qquad \ti M_{2r}(y):=\eta_{2r}(F^{-1}(y) )  M ( F^{-1}(y))  ,
$$
where $v_\e(x)=\e^{-1} w\left(\frac{x}{\e} \right).$ With these notations, \eqref{eq:TestFunction-Om-3D} becomes
\be  \label{eq:def-W-eps-3D}
\Psi_\e (y) = u_\e(y)+ \e^{\frac{1}{2}}  \textbf{c}\, \ti M_{2r}(y)= \e^{\frac{1}{2}}   V_\e(y)+    \e^{\frac{1}{2}}  \textbf{c}\, \ti M_{2r}(y) .
\ee
In the sequel we define $\mathcal{O}_{r, \e}$ as
$$
\lim_{r \to 0}\frac{\mathcal{O}_{r, \e}}{\e}=0. 
$$
Then we have the following.
\begin{lemma}\label{lem:expans-num-3D} 
We have 
\begin{align} \label{eq:Dirichlet-Psi-eps000}
\int_\O |\n \Psi_\e|^2 dy+  \int_{\O} h  |\Psi_\e|^2 dy=&\int_{\R^3} |\n w|^2 dx  +\pi \e \textbf{m}(y_0) \textbf{c}^2+\calO_r(\e),
\end{align}
as $\e \to 0$.
\end{lemma}
\begin{proof}
Recalling \eqref{eq:def-W-eps-3D}, 
direct computations give
\begin{align} \label{eq:nab-Psi}
\int_{F({Q}_{2r})\setminus F\left(\B_r\right)} |\n \Psi_\e|^2 dy&= \int_{F({Q}_{2r}) \setminus F\left(\B_r\right)} |\n \left( \ti \eta_r u_\e\right)|^2 dy+\e \textbf{c} ^2  \int_{F({Q}_{2r})\setminus F\left(\B_r\right)} |\n  \ti M_{2r} |^2 dy \nonumber\\
&+2 \e^{1/2}  \textbf{c}  \int_{F({Q}_{2r}) \setminus F\left(\B_r\right)} \n \left( \ti \eta_r u_\e\right) \cdot \n \ti M_{2r} dy \nonumber\\
 &= \e \int_{F({Q}_{2r}) \setminus F\left(\B_r\right)} |\n \left( \ti \eta_r V_\e\right)|^2 dy+\e  \textbf{c}^2   \int_{F({Q}_{2r}) \setminus F\left(\B_r\right)} |\n \ti  M_{2r} |^2 dy \nonumber\\
&+2 \e  \textbf{c}  \int_{F({Q}_{2r})\setminus F\left(\B_r\right)} \n \left( \ti \eta_r V_\e\right) \cdot \n \ti M_{2r} dy. 
\end{align}
By \eqref{eq:def-cut-off-cylind}, $\eta_r v_\e= \eta_r \e^{-1}w(\cdot/\e)$ is cylindrically symmetric. Therefore by the change   variable $y=F(x)$ and  using \eqref{Metric}, we get
\begin{align}\label{eq:enev}
\e \int_{F({Q}_{2r}) \setminus F\left(\B_r\right)} |\n \left( \ti \eta_r V_\e\right)|^2 dy&= \e \int_{{Q}_{2r} \setminus  \B_r } |\n \left(  \eta_r {v}_\e\right)|^2_g \sqrt{g} dx \nonumber\\
&= \e \int_{{Q}_{2r} \setminus  \B_r} |\n \left(  \eta_r {v}_\e\right)|^2 dx+O\left(\e  r^2 \int_{{Q}_{2r} \setminus  \B_r} |\n \left(  \eta_r {v}_\e\right)|^2 dx \right).
\end{align}
By computing, we find that 
\begin{align*}
\e  \int_{{Q}_{2r} \setminus  \B_r} |\n \left(  \eta_r {v}_\e\right)|^2 dx& \leq \e \int_{{Q}_{2r} \setminus  \B_r}   |\n  {v}_\e |^2 dx+\e  \int_{{Q}_{2r} \setminus  \B_r} v_\e^2 |\n \eta_r  |^2 dx + 2 \e \int_{{Q}_{2r} \setminus  \B_r} v_\e  |\n v_\e|| \n\eta_r|   dx  \nonumber\\
& \leq \e\int_{{Q}_{2r} \setminus  \B_r}   |\n  {v}_\e |^2 dx+ \frac{C}{r^2} \e\int_{{Q}_{2r} \setminus  \B_r} v_\e^2  dx + \frac{C}{r}  \e \int_{{Q}_{2r} \setminus  \B_r} v_\e  |\n v_\e|   dx \nonumber\\
&=\int_{\B_{2r/\e} \setminus  \B_{r/\e}}   |\n w |^2 dx+C \frac{\e}{r^2} \int_{\B_{2r/\e} \setminus  \B_{r/\e}}w^2  dx + \frac{C}{r}  \e \int_{\B_{2r/\e} \setminus  \B_{r/\e}}  w  |\n w|   dx.
\end{align*}
From  this and \eqref{eq:est-nw-sq-3D} and \eqref{eq:est-wnw-sq-3D}, we get 
$$
O\left(\e  r^2 \int_{{Q}_{2r} \setminus  \B_r} |\n \left(  \eta_r {v}_\e\right)|^2 dx \right)= \calO_r(\e).
$$
We replace  this in \eqref{eq:enev} to  have 
\begin{align}\label{eq:int-nv-et-v}
\e \int_{F({Q}_{2r}) \setminus F\left(\B_r\right)} |\n \left( \ti \eta_r  V_\e\right)|^2 dy&= \e \int_{{Q}_{2r} \setminus  \B_r}   |\n (\eta_r {v}_\e) |^2 dx + \calO_r(\e).
\end{align}
We have the following estimates
\be \label{eq:est-nv-eps}
0\leq v_\e\leq C |x|^{-1}\quad  \textrm{ for $x\in \R^3\setminus\{0\}$ }\qquad  \textrm{  and  }\qquad |\n v_\e(x)| \leq C |x|^{-2} \quad\textrm{ for $|x|\geq \e$, }
\ee
which easily follows from \eqref{eq:up-low-bound-w-3D}, \eqref{Metric}  and \eqref{eq:Green-R3-3D}.  
By these estimates, \eqref{Metric}, \eqref{DeterminantMetric} and \eqref{eq:regul-beta} together with  the change of variable $y=F(x)$, we have 
\begin{align*} 
\e \int_{F({Q}_{2r})\setminus F\left(\B_r\right)} \n \left( \ti \eta_r V_\e\right) \cdot \n \ti M_{2r} dy=& \e  \int_{ {Q}_{2r}\setminus  \B_{r} } \n \left(  \eta_{ {r} }  v_\e \right) \cdot \n   M   dx\\
&+ O\left( \e  \int_{ {\B}_{2r}\setminus  \B_{r} }|\n v_\e| dx   + \frac{\e}{r}   \int_{ {\B}_{2r}\setminus  \B_{r} } v_\e dx \right)\nonumber\\
= &\e  \int_{ {\B}_{2r}\setminus  \B_{r} } \n \left(  \eta_{ {r} }  v_\e \right) \cdot \n   M  dx + \calO_r(\e).
\end{align*}
This  with \eqref{eq:int-nv-et-v}, \eqref{eq:regul-beta} and \eqref{eq:nab-Psi}  give 
\begin{align*}
\int_{F({Q}_{2r})\setminus F\left(\B_r\right)} |\n \Psi_\e|^2 dy&= \e \int_{ {Q}_{2r} \setminus  \B_r } |\n \left(  \eta_r  {v}_\e\right)|^2 dx+\e  \textbf{c}^2   \int_{ {Q}_{2r}  \setminus  \B_r } |\n (\eta_{2r} M )|^2 dx\\
&+2 \e  \textbf{c}  \int_{{Q}_{2r}\setminus  \B_r } \n \left(  \eta_r  {v}_\e\right) \cdot \n  M  dx+ \calO_r(\e). 
\end{align*}  
Thanks  to Lemma \ref{lem:v-to-cR} and \eqref{eq:est-nv-eps}, we can thus use the dominated convergence theorem to deduce that, as $\e\to 0$, 
\be \label{eq:claim1-est-num}
\int_{ {Q}_{2r} \setminus  \B_r } |\n \left(  \eta_r  {v}_\e\right)|^2 dx= \textbf{c}^2\int_{ {Q}_{2r} \setminus  \B_r } |\n \left(  \eta_r  \cR\right)|^2 dx+o(1).
\ee
Similarly, we easily see that 
$$
 \int_{{Q}_{2r}\setminus  \B_r } \n \left(  \eta_r  {v}_\e\right) \cdot \n  M  dx= \textbf{c} \int_{{Q}_{2r}\setminus  \B_r } \n \left(  \eta_r  \cR \right) \cdot \n  M  dx+o(1)\qquad\textrm{ as $\e\to 0$.}
$$
This and  \eqref{eq:claim1-est-num}, then give
\begin{align}
\int_{F({Q}_{2r})\setminus F\left(\B_r\right)} |\n \Psi_\e|^2 dy 
&= \e\textbf{c}^2 \int_{ {Q}_{2r} \setminus  \B_r } |\n \left(  \eta_r  \cR\right)|^2 dx+\e  \textbf{c}^2   \int_{ {Q}_{2r}  \setminus  \B_r } |\n   M |^2 dx \nonumber\\
&+2 \e  \textbf{c}^2  \int_{{Q}_{2r}\setminus  \B_r } \n \left(  \eta_r  \cR \right) \cdot \n  M  dx+ \calO_r(\e) \nonumber\\
&=\e\textbf{c}^2 \int_{{Q}_{2r}\setminus  \B_r }  | \n ( \eta_r \cR+ M )|^2 dx+  \calO_r(\e). \label{eq:nPs-B2r-Br}
\end{align}
Since the support of $\Psi_\e$ is contained in $ \B_{4r}$ while the one of  $\eta_r$  is in ${Q}_{2r}$,  it is easy to deduce  from  \eqref{eq:regul-beta} that 
\begin{align*}
\int_{\O\setminus F\left({Q}_{2r}\right)} |\n \Psi_\e|^2 dy&= \e \textbf{c} ^2  \int_{F(\B_{4r})\setminus F\left({Q}_{2r}\right)} |\n  \ti M_{2r} |^2 dy=  \calO_r(\e)
\end{align*}
and from Lemma \ref{lem:estimates-for-3D}, that
$$
\int_{\O\setminus F\left(\B_r\right)}h | \Psi_\e|^2 dy=\e \textbf{c} ^2  \int_{F(\B_{4r})\setminus F\left(\B_{r}\right)} h | \eta_r V_\e +\ti M_{2r} |^2 dy=   \calO_r(\e).
$$
Therefore by \eqref{eq:nPs-B2r-Br}, we conclude that 
\begin{align*}
\int_{\O\setminus F\left(\B_r\right)} |\n \Psi_\e|^2 dy&+\int_{\O \setminus F\left(\B_r\right)}h | \Psi_\e|^2 dy \\
&\qquad=\e \textbf{c}^2  \int_{{Q}_{2r}\setminus  \B_r }  | \n ( \eta_r \cR+ M )|^2 dx+\e  \textbf{c}^2 \int_{{Q}_{2r}\setminus \B_r  }h(\cdot+y_0) | \eta_r \cR+ M  |^2 dx+ \calO_r(\e).
\end{align*}
Recall that $G(x+y_0,y_0)= \eta_r(x) \cR(x)+ M(x )$  for ever $x\in {Q}_{2r}$ and that  by \eqref{eq:Green-3D},   
 $$
 -\D_x G(x+y_0,y_0)+h(x+y_0) G(x+y_0,y_0)=0  \qquad \textrm{ for every  $x \in Q_{2r}\setminus Q_r $. }
 $$
 Therefore,   by integration by parts,   we find that
\begin{align*} 
\int_{\O\setminus F\left(\B_r\right)} |\n \Psi_\e|^2 dy+&\int_{\O\setminus F\left(\B_r\right)}h | \Psi_\e|^2 dy=\textbf{c}^2\int_{\de ({{Q}_{2r}\setminus \B_r )}} ( \eta_r \cR+ M ) \frac{\de ( \eta_r \cR+ M )}{\de \ov \nu}\s(x)+\calO_r(\e),
\end{align*} 
where $\ov \nu $  is the exterior normal vectorfield to ${Q}_{2r}\setminus \B_r  $. 
Thanks to   \eqref{eq:regul-beta}, we finally get 
\begin{align}\label{eq:nPsi-Ome-Br}
\int_{\O\setminus F\left(\B_r\right)} |\n \Psi_\e|^2 dy+&\int_{\O\setminus F\left(\B_r\right)}h | \Psi_\e|^2 dy 
&=-\e  \textbf{c}^2\int_{\de  \B_r } {\cR} \frac{\de {\cR}}{\de \nu} d\s(x)- \e \textbf{c}^2 \int_{\de  \B_r } M   \frac{\de {\cR}}{\de \nu} d\s(x) 
 + \calO_r(\e),
\end{align} 
where $\nu$ is  the exterior normal vectorfield to $  \B_r  $.\\
 Next we make the  expansion of $\int_{F\left(\B_r\right)} |\n \Psi_\e|^2 dy$ for $r$ and $\e$ small. First, we observe that, by  Lemma \ref{lem:estimates-for-3D} and \eqref{eq:regul-beta}, we have 
\begin{align*}
\int_{F\left(\B_r\right)}& |\n \Psi_\e|^2 dy =\int_{F\left(\B_r\right)} |\n u_\e|^2 dy+\e \textbf{c}^2 \int_{F\left(\B_r\right)} |\n M |^2 dy+2\e^{1/2}\textbf{c} \int_{F\left(\B_r\right)} \n u_\e \cdot \n \ti M_{2r} dy\\
&= \int_{ \B_{r/\e}} |\n  w|^2 dx +O\left( \e^2 \int_{ \B_{r/\e}} |x|^2 |\n  w|^2 dx +\e^{2}   \int_{   \B_{r/\e} } |\n  w |   dx \right) +\calO_r(\e) =  \int_{ \B_{r/\e}} |\n  w|^2 dx+ \calO_r(\e).
\end{align*}
By integration by parts and using \eqref{eq:est-L2-star-3D}, we deduce that 
\begin{align}\label{eq:expans-num-FQr}
\int_{F\left(\B_r\right)} |\n \Psi_\e|^2 dy 
&=\int_{\R^3} |\n w|^2 dx+ \int_{\de  \B_{r/\e}} w\frac{\de w }{\de \nu } d\s(x)+\calO_r(\e) \nonumber\\
&=\int_{\R^3} |\n w|^2 dx  + \e\int_{\de  \B_{r}} v_\e\frac{\de v_\e }{\de \nu } d\s(x) + \calO_r(\e). 
\end{align}
Now  \eqref{eq:est-nv-eps},   \eqref{eq:nv-eps-to-nv-C1} and the dominated convergence theorem yield, for  fixed $r>0$ and $\e\to 0$,
\begin{align}\label{eq:d-eps-nv-eps}
\int_{\de  \B_{r}} v_\e\frac{\de v_\e }{\de \nu } d\s(x)&=\int_{\de B^2_{\R^2}(0,r)}\int_{-r}^r v_\e(t,z)\n v_\e(t,z)\cdot\frac{z}{|z|} d\s(z) dt+2\int_{B^2_{\R^2}} v_\e(r,z)  \de _tv_\e(r,z) dz \nonumber\\
&= \textbf{c}^2\int_{\de B^2_{\R^2}(0,r)}\int_{-r}^r \cR(t,z)\n \cR(t,z)\cdot\frac{z}{|z|} d\s(z) dt+2 \textbf{c}^2\int_{B^2_{\R^2}} \cR(r,z) \de _t \cR(r,z) dz+o(1)\nonumber\\
& = \textbf{c}^2\int_{\de  \B_r } {\cR} \frac{\de {\cR}}{\de \nu} d\s(x) + o(1).
\end{align}
Moreover  \eqref{eq:est-nw-3D} implies that
\begin{align*}
  \int_{F(\B_{r})} h  \Psi_\e^2 dy  = \calO_r(\e).
\end{align*} 
From this together with \eqref{eq:expans-num-FQr} and \eqref{eq:d-eps-nv-eps}, we obtain
\begin{align*}
\int_{F\left(\B_r\right)} |\n \Psi_\e|^2 dy +  \int_{F(\B_{r})} h  \Psi_\e^2 dy
&=  \int_{\R^3} |\n w|^2 dx  +  \textbf{c}^2\e\int_{\de  \B_r } {\cR} \frac{\de {\cR}}{\de \nu} d\s(x)  +\calO_r(\e). 
\end{align*}
Combining this with \eqref{eq:nPsi-Ome-Br}, we then  have 
\begin{align} \label{eq:Dirichlet-Psi-eps}
\int_\O |\n \Psi_\e|^2 dy+  \int_{\O} h  \Psi_\e^2 dy=&\int_{\R^3} |\n w|^2 dx    - \e \textbf{c}^2\int_{\de \B_r } M \frac{\de {\cR}}{\de \nu} d\s(x)+\calO_r(\e) +o\left(\e\right).
\end{align}
Recalling that $\cR(x)=\frac{1}{|x|}$,  we have 
\begin{align*}
\int_{\de \B_r}\frac{\de \cR}{\de \nu}\, d\s(x)&=-  \int_{\de \B_r}\frac{x\cdot \nu (x)}{|x|^3}\, d\s(x)= \int_{ B_{\R^2}(0, r)}\frac{- 2 r   }{r^2+|z|^2}\, dz - 2\pi\int_{-r}^r\frac{r^3}{r^2+t^2}dt=- \pi^2 (1+ r^2)   .
\end{align*}
Since (recalling \eqref{eq:def-mass}) $M(y)=M(0)+O(r)=\textbf{m}(y_0)+O(r)$ in $\B_{2r}$, we get \eqref{eq:Dirichlet-Psi-eps000}. This then ends the proof.
\end{proof}
We finish by the following expansion
\begin{lemma}\label{lem:expans-denom-3D}
\begin{align*}
\frac{\l}{2^*_{s_1}}\int_{\O} \rho^{-s_1}_\G |\Psi_{\e}|^{2^*_{s_1}} dy&+\frac{1}{2^*_{s_2}}\int_{\O} \rho^{-s_2}_\G |\Psi_{\e}|^{2^*_{s_2}} dy= \frac{\l}{2^*_{s_1}}\int_{\R^3} |z|^{-s_1} |w|^{2^*_{s_1}} dx\\\\
&+\frac{1}{2^*_{s_2}}\int_{\R^3} |z|^{-s_2} |w|^{2^*_{s_2}} dx+{\e} \pi^2 \textbf{c}^2 \textbf{m}(y_0)+\calO_r(\e).
\end{align*}
\end{lemma}

\begin{proof}
Let $p>2$. Then there exists a positive constant $C(p)$  such that
\begin{equation*}
||a+b|^p-|a|^{p}-p ab |a|^{p-2}| \leq C(p) \left(|a|^{p-2} b^2+|b|^{p}\right)\qquad\textrm{  for all $a,b \in \R$.}
\end{equation*}
As a consequence, we obtain, for $s \in (0,2)$, that
\begin{align}\label{eq:expan-L-2-star-1}
\displaystyle\int_{\O} &\rho^{-s}_\G |\Psi_{\e}|^{2^*_s} dy= \displaystyle \int_{F(\B_{r})} \rho^{-s}_\G |u_{\e}+ \e^{\frac{1}{2}} \ti {M}_{2r}|^{2^*_s} dy+ \int_{F(\B_{4r}) \setminus F(\B_{r})} \rho^{-s}_\G |W_{\e}+ \e^{\frac{1}{2}} \ti {M}_{2r}|^{2^*_s} dy \nonumber\\
&=\displaystyle \int_{F(\B_{r})} \rho^{-s}_\G | u_{\e}|^{2^*_s} dy
+
2^*_s \textbf{c} {\e}^{1/2} \int_{F(\B_{r})} \rho^{-s}_\G | u_{\e}|^{2^*_s-1}   \ti M_{2r} dy \nonumber\\
&\quad  \displaystyle + O\left(\int_{F\left(\B_{ 4r}\right)}   \rho^{-s}_\G |\eta_r u_{\e}|^{2^*_s-2} \left({\e}^{1/2}\ti {M}_{2r} \right)^2 dy
+ \int_{F\left(\B_{ 4r}\right)} \rho^{-s}_\G |{\e}^{1/2} \ti {M}_{2r} |^{2^*_s} dy\right) \nonumber\\
&\quad  \displaystyle+O \left(  \int_{F(\B_{4r}) \setminus F(\B_{r})} \rho^{-s}_\G | u_{\e}|^{2^*_s} dy
+
2^*_s \textbf{c} {\e}^{1/2} \int_{F(\B_{4r}) \setminus F(\B_{r})} \rho^{-s}_\G | u_{\e}|^{2^*_s-1}   \ti {M}_{2r} dy   \right).
\end{align}
By H\"{o}lder's inequality and \eqref{DeterminantMetric}, we have
\begin{align}\label{eq:expan-L-2-star-2}
\int_{F\left(\B_{ 4r}\right)}   \rho^{-s}_\G |\eta u_{\e}|^{2^*_s-2} \left({\e}^{1/2}\ti \b_r \right)^2 dy& \leq \e \|u_{\e}\|_{L^{2^*_s}(F(\B_{4 r} );\rho^{-s})}^{{2^*_s-2}}\|\ti{M}_{2r}\|_{L^{2^*_s}(F(\B_{4 r} );\rho^{-s}_\G)}^{{2} }\nonumber\\
&=\e \|w\|_{L^{2^*_s}( \B_{ 4r};|z|^{-s}\sqrt{|g|} )}^{{2^*_s-2}}  \|\ti{M}_{2r}\|_{L^{2^*_s}(F(\B_{4 r} );\rho^{-s}_\G)}^{{2} }\nonumber\\
&\leq \e (1+ Cr)\|\ti{M}_{2r}\|_{L^{2^*_s}(F(\B_{ 4r} );\rho^{-s}_\G)}^{{2} }=\calO_r(\e).
\end{align}
Furthermore, since $2^*_s>2$, by \eqref{eq:regul-beta}, we easily get 
\begin{align} \label{eq:expan-L-2-star-2-00}
\int_{F\left(\B_{4 r}\right)} \rho^{-s}_\G |{\e}^{1/2} \ti {M}_{2r} |^{2^*_s} dy=o(\e).
\end{align}
Moreover  by change of variables and \eqref{eq:est-L2-star-3D}, we also have
\begin{align*}
\int_{F(\B_{4r}) \setminus F(\B_{r})} \rho^{-s}_\G | u_{\e}|^{2^*_s} dy
+
2^*_s \textbf{c} {\e}^{1/2}  &\int_{F(\B_{4r}) \setminus F(\B_{r})} \rho^{-s}_\G | u_{\e}|^{2^*_s-1}   \ti M_{2r} dy\\
     \leq C& \int_{\B_{4r/\e} \setminus  \B_{r/\e}} |z|^{-s}  | w|^{2^*_s} dx
+
C   {\e}   \int_{\B_{4r/\e} \setminus  \B_{r/\e} } |z|^{-s}  | w|^{2^*_s-1}    dx  =o(\e).
\end{align*}
By this, \eqref{eq:expan-L-2-star-1}, \eqref{eq:expan-L-2-star-2-00} and \eqref{eq:expan-L-2-star-2}, it results
\begin{align*}
\displaystyle\int_{\O} \rho^{-s}_\G |\Psi_{\e}|^{2^*_s} dy&=  \displaystyle \int_{F(\B_{r})} \rho^{-s}_\G | u_{\e}|^{2^*_s} dy
+
2^*_s \textbf{c} {\e}^{1/2} \int_{F(\B_{r})} \rho^{-\s}_\G | u_{\e}|^{2^*_\s-1}   \ti M_{2r} dy +\calO_r(\e).
\end{align*}
We define  $B_\e(x):=M(\e x)  \sqrt{|g_\e|}(x) =M(\e x)  \sqrt{|g|}(\e x)$. Then
by    the change of variable $y=\frac{F(x)}{\e}$ in the above identity and recalling \eqref{DeterminantMetric}, then by oddness, we have 
\begin{align*}
\displaystyle\int_{\O} \rho^{-s}_\G |\Psi_{\e}|^{2^*_s} dy &=  \displaystyle\int_{\B_{r/\e}} |z|^{-s}  w ^{2^*_s} \sqrt{|g_\e|}dx
+
2^*_s {\e}  \textbf{c}  \int_{\B_{r/\e}}|z|^{-s} | w |^{2^*_s-1} B_\e dx+ \calO_r(\e)\\
&= \displaystyle\int_{\B_{r/\e}} |z|^{-s}  w ^{2^*_s} dx
+
2^*_s {\e}  \textbf{c}  \int_{\B_{r/\e}}|z|^{-s} | w |^{2^*_s-1} B_\e dx+ \calO_r(\e)\\
&\quad \displaystyle+  O\left(   \e^2\int_{\B_{r/\e}} |z|^{-s} |x|^2  w ^{2^*_s}  dx\right)\\
&=\displaystyle \int_{\R^3} |z|^{-s} |w|^{2^*_s} dx + 2^*_s {\e}  \textbf{c}  \int_{\B_{r/\e}}|z|^{-s} | w |^{2^*_s-1} B_\e dx\\
&\quad \displaystyle+O\left( \int_{\R^3\setminus \B_{r/\e}} |z|^{-s}  w ^{2^*_s}  dx+ \e^2\int_{\B_{r/\e}} |z|^{-s} |x|^2  w ^{2^*_s}  dx\right)+ \calO_r(\e).
\end{align*}
By \eqref{eq:est-L2-star-3D} we then have 
\begin{align} \label{eq:est-L-2star-Psi-not-ok}
\int_{\O} \rho^{-s}_\G |\Psi_{\e}|^{2^*_s} dy =
   \displaystyle \int_{\R^3} |z|^{-s} |w|^{2^*_s} dx
+
2^*_s {\e}  \textbf{c}  \int_{\B_{r/\e}}|z|^{-s} | w |^{2^*_s-1} B_\e(x)  dx +\calO_r(\e).
\end{align}
Therefore for $0<s_2<s_1<2$, we have
\begin{align*}
\frac{\l}{2^*_{s_1}}\int_{\O} \rho^{-s_1}_\G |\Psi_{\e}|^{2^*_{s_1}} dy&+\frac{1}{2^*_{s_2}}\int_{\O} \rho^{-s_2}_\G |\Psi_{\e}|^{2^*_{s_2}} dy= \frac{\l}{2^*_{s_1}}\int_{\R^3} |z|^{-s_1} |w|^{2^*_{s_1}} dx+\frac{1}{2^*_{s_2}}\int_{\R^3} |z|^{-s_2} |w|^{2^*_{s_2}} dx\\\\
&+{\e}  \textbf{c} \l \int_{\B_{r/\e}}|z|^{-s_1} | w |^{2^*_{s_1}-1} B_\e(x)  dx+{\e}  \textbf{c}  \int_{\B_{r/\e}}|z|^{-s_2} | w |^{2^*_{s_2}-1} B_\e(x)  dx +\calO_r(\e).
\end{align*}

We multiply \eqref{A1A1} by $B_\e\in \cC^1(\overline{Q_r})$ and we integrate by parts to get 
\begin{align*}
\l \int_{\B_{r/\e}}|z|^{-s_1} | w |^{2^*_{s_1}-1}  B_\e    dx+\int_{\B_{r/\e}}|z|^{-s_2} | w |^{2^*_{s_2}-1}  B_\e    dx& = \int_{ \B_{r/\e} } \n w \cdot \n B_\e dx -\int_{ \de \B_{r/\e} }B_\e  \frac{ \de  w}{\de \nu}  d\s(x)\\
&= \int_{ \B_{r/\e} } \n w \cdot \n B_\e dx-\int_{ \de \B_{r} }B_1  \frac{ \de  v_\e}{\de \nu}  d\s(x).
\end{align*}
Since $|\n B_\e|\leq C \e$,  by Lemma \ref{lem:v-to-cR} and \eqref{eq:regul-beta}, we then have 
$$
\e   \int_{ \B_{r/\e} } \n w \cdot \n B_\e dx  =O\left( \e^2  \int_{ \B_{r/\e} } |\n w| dx \right)= \calO_r(\e). 
$$
Consequently, on the one hand,
\begin{align*}
\l \e \int_{\B_{r/\e}}|z|^{-s_1} | w |^{2^*_{s_1}-1}  B_\e    dx+\e \int_{\B_{r/\e}}|z|^{-s_2} | w |^{2^*_{s_2}-1}  B_\e    dx
&=  -\e\int_{ \de \B_{r} }B_1  \frac{ \de  v_\e}{\de \nu}  d\s(x)+ \calO_r(\e).
\end{align*}
On the other hand    by Lemma \ref{lem:v-to-cR}, \eqref{eq:regul-beta}  and the dominated convergence theorem, we get
\begin{align*}
 \int_{ \de \B_{r} } B_1  \frac{\de  v_\e}{\de \nu}    d\s(x)= \textbf{c}  \int_{ \de \B_{r} } B_1  \frac{\de  \cR}{\de \nu}    d\s(x)+o(1)= \textbf{c} {M}(0) \int_{ \de \B_{r} }    \frac{\de  \cR}{\de \nu}    d\s(x)+O(r)+o(1),
\end{align*}
so that 
\begin{align*}
\l \e c\int_{\B_{r/\e}}|z|^{-s_1} | w |^{2^*_{s_1}-1}  B_\e    dx+\e c\int_{\B_{r/\e}}|z|^{-s_2} | w |^{2^*_{s_2}-1}  B_\e    dx
&=  -\e  \textbf{c}^2 M(0)\int_{ \de \B_{r} }   \frac{ \de  \cR}{\de \nu}  d\s(x)+ \calO_r(\e).
\end{align*}
It then follows from \eqref{eq:est-L-2star-Psi-not-ok} that
\begin{align*}
\frac{\l}{2^*_{s_1}}\int_{\O} \rho^{-s_1}_\G |\Psi_{\e}|^{2^*_{s_1}} dy&+\frac{1}{2^*_{s_2}}\int_{\O} \rho^{-s_2}_\G |\Psi_{\e}|^{2^*_{s_2}} dy= \frac{\l}{2^*_{s_1}}\int_{\R^3} |z|^{-s_1} |w|^{2^*_{s_1}} dx\\\\
&+\frac{1}{2^*_{s_2}}\int_{\R^3} |z|^{-s_2} |w|^{2^*_{s_2}} dx-{\e}  \textbf{c}^2 {M}(0) \int_{ \de \B_{r} } \frac{\de  \cR}{\de \nu}    d\s(x) +\calO_r(\e).
\end{align*}
Finally, recalling that $\cR(x)=\frac{1}{|x|}$,  we have 
\begin{align*}
\int_{\de \B_r}\frac{\de \cR}{\de \nu}\, d\s(x)&=-  \int_{\de \B_r}\frac{x\cdot \nu (x)}{|x|^3}\, d\s(x)= \int_{ B_{\R^2}(0, r)}\frac{- 2 r   }{r^2+|z|^2}\, dz - 2\pi\int_{-r}^r\frac{r^3}{r^2+t^2}dt=- \pi^2 (1+ r^2)   .
\end{align*}
Since ${M} (0)=\textbf{m}(y_0)$, see \eqref{eq:def-mass}, the proof of the lemma is thus finished.
\end{proof}
Now we are in position to complete the proof of our main result.
\begin{proof} \textbf{of Theorem \ref{th:main1}}\\
Combining  Lemma \ref{lem:expans-num-3D} and Lemma \ref{lem:expans-denom-3D} and recalling \eqref{Functional} and \eqref{Functional1}, we have
\begin{align} 
J\left(t u_\e\right)=\Psi(t w)+\mathcal{M}_{r, \e}(t w),
\end{align}
for some function $\mathcal{M}: \calD^{1, 2}(\R^N) \to \R$ satisfying
$$
\mathcal{M}_{r, \e}(w)=-\frac{\e}{2} c^2 \pi^2  m(y_0)+\mathcal{O}_{r, \e}.
$$
Since $2^*_{s_2}> 2^*_{s_1}$, $\Psi(tu_\e)$ has a unique maximum, we have
$$
\max_{t\geq 0} \Psi(tw)=\Psi(w)=\b^*.
$$
Therefore, the maximum of $J(t u_\e)$ occurs at $t_\e:=1+o_\e(1)$. 
Thanks to assumption \eqref{ExistenceAssumption}, we have
$$
\mathcal{M}_{r, \e}(w)<0.
$$
Therefore
$$
\max_{t \geq 0} J(t u_\e):= J(t_\e u_\e)\leq \Psi(t_\e w)+\e^2 \mathcal{G}(t_\e w) \leq \Psi(t_\e w) < \Psi(w)=\b^*.
$$
We thus get the desired result.
\end{proof}
%
%
%
%
%
%

\end{document}